\documentclass[11pt]{amsart}

\usepackage{amsthm}
\usepackage{amsmath}
\usepackage{amssymb}
\usepackage{color}

\newtheorem{thm}{Theorem}[section]
\newtheorem{theorem}[thm]{Theorem}

\newtheorem{corollary}[thm]{Corollary}

\newtheorem{observation}[thm]{Observation}

\newtheorem{lemma}[thm]{Lemma}
\newtheorem{proposition}[thm]{Proposition}

\newtheorem{definition}[thm]{Definition}

\theoremstyle{remark}

\newcommand{\RR}{\mathbb R}

\newcommand{\CC}{\mathbb C}
\newcommand{\HH}{\mathbb H}

\newcommand{\ip}[2]{\left\langle#1,#2\right\rangle}

\newcommand{\absip}[2]{\left| \left\langle#1,#2\right\rangle \right|}

\usepackage{enumerate}

\DeclareMathAlphabet{\mathpzc}{OT1}{pzc}{m}{it}
\DeclareMathOperator{\spn}{span}

\renewcommand\Re{\operatorname{Re}}
\renewcommand\Im{\operatorname{Im}}

\begin{document}

\title{phase Retrieval verses phaseless reconstruction}
\author[Botelho-Andrade, Casazza, Nguyen, Tremain
 ]{Sara Botelho-Andrade, Peter G. Casazza, Hanh Van Nguyen, and Janet C. Tremain}
\address{Department of Mathematics, University
of Missouri, Columbia, MO 65211-4100}

\thanks{The authors were supported by
 NSF DMS 1307685; and  NSF ATD 1042701 and 1321779; AFOSR  DGE51:  FA9550-11-1-0245}

\email{Casazzap@missouri.edu; hnc5b@mail.missouri.edu}
\email{ Tremainjc@missouri.edu}

\begin{abstract}
In 2006, Balan/Casazza/Edidin \cite{BCE} introduced the frame theoretic study of phaseless reconstruction.  Since then, this has turned into a very active area of research.  Over the years, many people have replaced the term {\it phaseless reconstruction} with {\it phase retrieval}.  Casazza then asked:  {\it Are these really the same?}  In this paper, we will show that phase retrieval is equivalent
to phaseless reconstruction.  We then show, more generally, that phase retrieval by projections is equivalent to phaseless reconstruction by projections.  Finally, we study {\it weak
phase retrieval} and discover that it is very different from
phaseless reconstruction.
\end{abstract}

\maketitle

\section{Introduction}
Phase retrieval is an old problem in signal processing and has been studied for over 100 years by electrical engineers.
Let $x=(a_1,a_2,\ldots,a_m)$ and $y=(b_1,b_2,\ldots,b_m)$
be vectors in $\HH_m$.  We say that $x,y$ have the {\bf same phases} if
\[ \mbox{phase }a_i = \mbox{phase }b_i,\mbox{ for all }
i=1,2,\ldots,m.\]

\begin{definition}
Let $\Phi=\{\phi_i\}_{i=1}^n$ be a family of vectors in $\HH_m$ (resp. $\{P_i\}_{i=1}^n$ is a family of projections on $\HH_m$) satisfying:  for every
$x=(a_1,a_2,\ldots,a_m)$ and $y=(b_1,b_2,\ldots b_m)$ and
\begin{equation}\label{E1} \absip{x}{\phi_i}=\absip{y}{\phi_i},
\mbox{ for all } i=1,2,\ldots,n.
\end{equation}
Respectively,
\begin{equation}\label{E2} \|P_ix\|=\|P_iy\|, \mbox{ for all }i=1,2,\ldots,n.
\end{equation}
\begin{enumerate}
\item If this
implies there is a $|\theta|=1$ so that $x$ and $\theta y$
have the same phases,
  we say $\Phi$ does {\bf phase retrieval} (Respectively, $\{P_i\}_{i=1}^n$ does
{\bf phase retrieval}).  Moreover, in the real case, if
$\theta =1$ we say $x$ and $y$ have the {\bf same signs}
and if $\theta=-1$ we say $x$ and $y$ have {\bf opposite
signs}.
\item If this implies there is a $|\theta|=1$ so that
$x=\theta y$, we say $\Phi$ (does {\bf phaseless reconstruction.} (Respectively, $\{P_i\}_{i=1}^n$ does
{\bf phaseless reconstruction}.)
\end{enumerate}
\end{definition}

In the setting of frame theory,
the concept of phaseless reconstruction was
introduced in 2006 by Balan/Casazza/Edidin \cite{BCE}.
At that time, they showed that in the real case, a
{\it generic} family of (2m-1)-vectors in $\RR_m$ does
phaseless reconstruction and no set of (2m-2)-vectors can do
this.  In the complex case, they showed that a {\it generic}
set of (4m-2)-vectors does phaseless reconstruction.  Heinossaari, Mazzarella and Wolf \cite{HMW} show that
n-vectors doing phaseless reconstruction in $\CC_m$ requires
$n\ge 4m-4-2\alpha$, where $\alpha$ is the number of $1's$
in the binary expansion of (m-1).  Bodmann \cite{B} showed
that phaseless reconstruction in $\CC_m$ can be done with $(4m-4)$-vectors.  Later, Conca, Edidin, Hering, and Vinzant
\cite{CEHV} that a {\it generic} frame with $(4m-4)$-vectors does phaseless reconstruction in $\CC_m$.  They also show
that if $m=2^k+1$ then no $n$-vectors with $n< 4m-4$ can
do phaseless reconstruction.  Bandeira, Cahill, Mixon, and Nelson \cite{BCMN} conjectured that for all $m$, no fewer than
(4m-4)-vectors can do phaseless reconstruction.  Recently,
Vinzant \cite{V} showed that this conjecture does not hold by giving 11 vectors in $\CC_4$ which do phaseless reconstruction.

Over the years, we have started replacing the phrase:
{\it phaseless reconstruction} with the phrase: {\it phase
retrieval}.  Casazza at a meeting in 2012 raised the question:  {\it Are these really the same?}  In this paper
we will answer this question
in the affirmative, and the same for phase retrieval by projections, and
then show that the notion of {\it weak phase retrieval} is not equivalent to phaseless reconstruction.

The problem occurred here because of the way we translated the engineering version of {\it phase retrieval} into the language of frame theory.  The engineers are working with
the modulus of the Fourier transform and want to recover the phases so they can invert the Fourier transform to discover the signal.  So all they need to do is to recover the phase.
But in the frame theory version of this, for $x=(a_1,a_2,\ldots,a_m)$ we are really trying to recover two things:
\begin{enumerate}
\item Recover the phases of the $a_i$.
\item Recover $|a_i|$ (which in the engineering case, is
already known).
\end{enumerate}

For notation, we will use $\HH_m$ to denote a real or complex $m$-dimensional Hilbert space and for the real and
complex cases we use $\RR_m$ and $\CC_m$ respectively.

\section{Phase Retrieval verses Phaseless Reconstruction}

We will need the {\it complement property} from \cite{BCE}.

\begin{definition}
A family of vectors $\{\phi_i\}_{i=1}^n$ in $\HH_m$ has
the {\bf complement property} if for every $I\subset [n]$,
either $\{\phi_i\}_{i\in I}$ spans $\HH_m$ or $\{\phi_i\}_{i\in I^c}$ spans $\HH_m$.
\end{definition}

We will prove that phase retrieval implies the complement
property for both the real and complex cases and even
in a more general setting.
First, we need to make a couple observations:

\begin{observation}\label{o1}
If $\{\phi_i\}_{i=1}^n$ does phase retrieval in $\HH_m$,
then span $\{\phi_i\}_{i=1}^n = \HH_m$.  For otherwise,
there would exist $0\not= x\in \HH_m$ so that
\[ \langle x,\phi_i\rangle = \langle 0,\phi_i\rangle =0,
\mbox{ for all } i=1,2,\ldots,n,\]
while $x,0$ do not have the same phase.
\end{observation}

\begin{observation}\label{o2}
If $x=(a_1,a_2,\ldots,a_m)$ and $y=(b_1,b_2,\ldots,b_m)$
 have the same phases, then $a_i = 0$ if and only if
 $b_i=0$.  I.e.  zero has no phase.
\end{observation}

We need a result from \cite{BCE}.  This result was proved
in \cite{BCE} for the real case and it was stated that the
same proof works in the complex case.  In \cite{BCMN} they
state that this claim is {\it erroneous}, and this proof
does not work in the complex case.  But, the proof in
\cite{BCE} does work in the complex case and is much easier than the one supplies in \cite{BCMN}.  So, we
will now show that this argument does in fact work in the
coomplex case.

\begin{theorem}[Balan/Casazza/Edidin]\label{t1}
Let $\Phi=\{\phi_i\}_{i=1}^n$ be vectors in $\HH_m$.   If
$\Phi=\{\phi_i\}_{i=1}^n$ does phaseless reconstruction, then it has complement property.  Moreover, in the real case, these are equivalent while in the complex case they
are not equivalent.
\end{theorem}

\begin{proof}
Assume $\Phi$ fails complement property but does phaseless reconstruction.  Choose $I\subset [n]$ so that neither of the sets of vectors $\{\phi_i\}_{i\in I}$ or $\{\phi_i\}_{i\in I^c}$ spans $\HH_m$.  Choose
$\|x\|=1=\|y\|$ so that $x \perp \phi_i$ for $i\in I$
and $y\perp \phi_i$ for $i\in I^c$.  Then
\[ |\langle x+y,\phi_i\rangle|=|\langle x-y,\phi_i\rangle|,
\mbox{ for all }i=1,2,\ldots,n.\]
Since $\Phi$ does phaseless reconstruction, there is a
$|\theta|=1$ so that
\[ x+y=\theta(x-y), \mbox{ and hence }
(1-\theta)x= -(1+\theta)y.\]
If $\theta = 1$, then $y=0$ and if $\theta = -1$ then
$x=0$, contradicting the fact that $x,y$ are unit norm.
Otherwise,
\[ x = \frac{-(1+\theta)}{1-\theta}y = dy,\mbox{ for }
d \not= 0.\]
Now,
\[ \langle x,\phi_i\rangle = \langle dy,\phi_i\rangle
=0,\mbox{ for all }i=1,2,\ldots,n,\]
and hence $\Phi$ does not span $\HH_m$ contradicting
Observation \ref{o1}.
\end{proof}

\begin{theorem}\label{t3}
Let $\{P_i\}_{i=1}^n$ be projections onto the subspaces
$\{W_i\}_{i=1}^n$ of $\HH_m$ which do phase retrieval. Then
\vskip12pt
 For every orthonormal basis $\{\phi_{i,j}\}_{j=1}^{D_i}$ of $W_i$, the set $\{\phi_{i,j}\}_{i=1, j=1}^{n,\ D_i}$ has complement property.
\end{theorem}

\begin{proof}
Suppose $\{W_i\}_{i=1}^n$ does phase retrieval for $\HH_m$, but fails phaseless reconstruction. By Theorem \ref{t1}, there exist an orthonormal basis $\{\phi_{i,j}\}_{j=1}^{D_i}$ of each $W_i$ such that the set $\{\phi_{i,j}\}_{i=1, j=1}^{n,\ D_i}$ fails the complement property. In other words, there exists $I\subset \{ (i,j): 1\leq i\leq n\text{ and} 1\leq j \leq D_i\}$ so that $\{\phi_{i,j}\}_{(i,j)\in I}$ and $\{\phi_{i,j}\}_{(i,j)\in I^c}$ do not span $\HH_m$. Choose vectors $x,y\in \HH_m$ with $\|x\|=1=\|y\|$,
and $x=(a_1,a_2,\ldots,a_m)$ and $y=(b_1,b_2,\ldots,b_m)$
such that $x\perp \phi_{i,j}$ for all $(i,j)\in I$ and $y\perp \phi_{i,j}$ for all $(i,j)\in I^c$. Note this choice of vectors forces that for each $(i,j)$ either $\ip{x}{\phi_{i,j}}=0$ or $\ip{y}{\phi_{i,j}}=0$.  Fix $0\not= c$.
Then for each $1\leq i \leq n$
\[ |\langle x+cy,\phi_{ij}\rangle| =|\langle x-cy,\phi_{ij}\rangle|,\mbox{ for all }i,j.\]  Hence,
\[\|P_i(x+cy)\|^2=\sum_{j=1}^{D_i} \absip{x+cy}{\phi_{i,j}}^2=\sum_{j=1}^{D_i} \absip{x-cy}{\phi_{i,j}}^2=\|P_i(x-cy)\|^2.\]
By assumption that $\{P_i\}_{i=1}^n$ does phase retrieval, this implies there is a $|\theta|=1$ so that $x+cy$ and
$\theta(x-cy)$ have the same phases.  Assume there exists
some $1\le i_0\le m$ so that $a_{i_0}\not= 0 \not= b_{i_0}$
and let $c = \frac{-a_{i_0}}{b_{i_0}}$.  Then
\[ (x+cy)_{i_0} = a_{i_0}+cb_{i_0}= a_{i_0}+\frac{-a_{i_0}}{b_{i_0}}b_{i_0}=0,\]
while
\[ (x-cy)_{i_0} = a_{i_0}-\frac{-a_{i_0}}{b_{i_0}}=2a_{i_0}
\not= 0.\]
But this contradicts Obeservation \ref{o2}.  It follows that
for every $1\le i \le m$, either $a_i=0$ or $b_i=0$.  Let
$\{e_i\}_{i=1}^m$ be an orthonormal basis for $\HH_m$ and
let $I=\{1\le i \le m:b_i=0\}$.  Then
\[ x+y = \sum_{i\in I}a_ie_i +\sum_{i\in I^c}b_ie_i,
\mbox{ and }x-y = \sum_{i\in I}a_ie_i + \sum_{i\in I^c}(-b_i)e_i.\]
By the above argument, there is a $|\theta|=1$ so that
$x+y$ and $\theta(x-y)$ have this same phase.  But this
is clearly impossible.  This final contradiction completes
the proof.
\end{proof}

We have a number of consequences of Theorem \ref{t3}.
Letting the subspaces $W_i$ be one dimensional, this
becomes a theorem about vectors.

\begin{corollary}
If $\Phi = \{\phi_i\}_{i=1}^n$ does phase retrieval in
$\HH_m$, then $\Phi$ has the complement property.  Hence,
in the real case, phase retrieval and phaseless reconstruction are equivalent properties.
\end{corollary}

We need a result from \cite{CCPW}.

\begin{theorem}\label{t5}
Let $\{P_i\}_{i=1}^n$ be projections onto the subspaces
$\{W_i\}_{i=1}^n$ of $\HH_m$.  The following are equivalent:
\begin{enumerate}
\item $\{P_i\}_{i=1}^n$ does phaseless reconstruction.
\item For every orthonormal basis $\{\phi_{i,j}\}_{j=1}^{D_i}$ of $W_i$, the set $\{\phi_{i,j}\}_{i=1, j=1}^{n,\ D_i}$ does phaseless reconstruction.
\end{enumerate}
\end{theorem}

Combining Theorems \ref{t3}, \ref{t5}:

\begin{corollary}
In $\RR_m$, a family of projections $\{P_i\}_{i=1}^n$ does phase retrieval if and only if it does phaseless reconstruction.
\end{corollary}

In the complex case, the complement property is not equivalent to phaseless reconstruction.  We will
show that phase retrieval and phaseless reconstruction
are equivalent in the complex case in the next section.

\section{Complex Case}

\begin{theorem}\label{BCMN}\cite{abandeira13} Consider $\Phi=\lbrace \phi_n\rbrace _{n=1}^N\subseteq \CC^M$ and the mapping $\mathpzc{A}:\CC^M/\mathbb{T}\rightarrow\RR^N$ defined by $(\mathpzc{A}(x))(n):=|\langle x,\phi_n\rangle|^2$. Viewing $\lbrace\phi_n\phi_n^*u\rbrace_{n=1}^N$ as vectors in $\RR^{2M}$, denote $S(u):=\spn _\RR \lbrace \phi_n\phi_n^*u\rbrace_{n=1}^N$. Then the following are equivalent:
\begin{enumerate}
\item[(a)] $\mathpzc{A}$ is injective.
\item[(b)] $\dim S(u)\geq 2M-1$ for every $u\in \CC^M\setminus\lbrace 0\rbrace$.
\item[(c)] $S(u)=\spn_{\RR}\lbrace iu\rbrace^\perp$ for every $u\in\CC^M\setminus\lbrace 0 \rbrace$.
\end{enumerate}
\end{theorem}
For this section we adopt the notation $\langle a,b\rangle_{\RR}$ to denote $\Re\langle a,b\rangle$.
\begin{lemma} \label{Lem5.2} Given $\lbrace \phi_n\rbrace _{n=1}^N\subseteq \CC^M$ and any $u\in\CC^M$ then $\langle \phi_n\phi_n^*u, iu\rangle_\RR=0$
\end{lemma}
\begin{proof} The following calculation gives the result almost immediately:
\begin{align*}
\langle \phi_n\phi_n^*u, iu\rangle_\RR=&\langle\langle u,\phi_n\rangle \phi_n, iu\rangle_\RR=\Re(-i\langle u, \phi_n\rangle\langle \phi_n,u\rangle)
 \\=&-\Re(i|\langle u,\phi_n\rangle|^2)=0.
\end{align*}
\end{proof}
\begin{lemma}\label{Lem5.3} Given $\lbrace \phi_n\rbrace _{n=1}^N\subseteq \CC^M$ and any $u,v\in\CC^M$ then for each $\phi_n$,
\[|\langle u+v,\phi_n\rangle|^2-|\langle u-v,\phi_n\rangle|^2=4\langle \phi_n\phi_n^*u,v\rangle_\RR.\]
\end{lemma}
\begin{proof}
Consider the following
\begin{equation}\label{eq1}
|\langle u+v,\phi_n\rangle|^2=|\langle u,\phi_n\rangle|^2+2\Re(\langle u,\phi_n\rangle\overline{\langle v,\phi_n\rangle})+|\langle v,\phi_n\rangle|^2
\end{equation}
and
\begin{equation}\label{eq2}
|\langle u-v,\phi_n\rangle|^2=|\langle u,\phi_n\rangle|^2-2\Re(\langle u,\phi_n\rangle\overline{\langle v,\phi_n\rangle})+|\langle v,\phi_n\rangle|^2.
\end{equation}
Then subtracting (\ref{eq2}) from (\ref{eq1}) we obtain
\[|\langle u+v,\phi_n\rangle|^2-|\langle u-v,\phi_n\rangle|^2=4\Re(\langle u,\phi_n\rangle\overline{\langle v, \phi_n\rangle})=4\langle \phi_n\phi_n^*u,v\rangle_\RR\]
\end{proof}
\begin{corollary} If $\lbrace \phi_n\rbrace_{n=1}^N$ does phaseless reconstruction and $\langle\phi_n\phi_n^* u, v\rangle_\RR=0$ for each $n$ then $u+v=\omega (u-v)$ for $|\omega|=1$ and thus $v=\frac{2\Im (\omega)}{|1+\omega|^2}u$.
\end{corollary}
\begin{proof}
If $u+v=\omega u-\omega v$ then $v=\frac{\omega - 1}{\omega +1}u=-\frac{(1-\omega)(1+\overline{\omega})}{|1+\omega|^2}u=\frac{2\Im (\omega)}{|1+\omega|^2}u.$
\end{proof}
\begin{lemma} Given any $u$, let $v=\alpha iu$ for $\alpha\in\RR$ and let $\omega=\frac{1+\alpha i}{1-\alpha i}$ then $|\omega|=1$ and $u+v=u(1+\alpha i )=\frac{1+\alpha i }{1-\alpha i}(u-\alpha i u)=\omega (u-v)$.
\end{lemma}
\begin{lemma} If $x-y\neq 0$ then $\langle \phi\phi^*(x-y),x+y\rangle _\RR=0$.
\end{lemma}
\begin{proof} Consider the following calculation,
\begin{align*}
\langle \phi\phi^*(x-y),x+y\rangle _\RR&=\Re((x+y)^*\phi\phi^*(x-y))\\
&=\Re(|\phi^*x|^2-x^*\phi\phi^*y+y^*\phi\phi^*x-|\phi y|^2)\\
 &=\Re(-x^*\phi\phi^*y+x^*\phi\phi^*y)=0.
\end{align*}
\end{proof}

\begin{lemma} \label{Lem5.11} Let $a,b\in\CC$ such that $|a|+|b|>0$. If $$\arg (a+b)=\arg (e^{i\theta} (a-b)),$$  then
\[\tan \theta=\frac{2 \Im (\bar{a}b)}{|a|^2-|b|^2}\]
for $|a|\neq |b|$ and $\theta=\pi/2$ otherwise.
\end{lemma}
\begin{theorem}
Phase retrieval implies phaseless construction in the complex case.
\end{theorem}
\begin{proof}
Suppose $\Phi=\lbrace \phi_n\rbrace _{n=1}^N\subseteq \CC^M$ does phase retrieval. Let $u,v$ be non-zero vectors in $\CC^M$ such that $\langle \phi_n\phi^*_n u,v\rangle_\RR=0$ for all $n$. Note that Lemma \ref{Lem5.3} ensures that $|\langle u+v,\phi_n\rangle |^2=|\langle u-v,\phi_n\rangle |^2$ for each $n$. To apply the results in Theorem \ref{Thm5.1}, we must show $v=\lambda i u$ for some $\lambda\in\RR$.  For simplicity, denote $u=(u_1, u_2,...)$ and $v=(v_1,v_2,...)$. Consider the following cases:

\vskip12pt
\noindent {\bf Case 1}:  $u_jv_j=0$ for all $1\leq j\leq N$.
\vskip12pt
Without loss of generality, suppose $u=(e^{i\alpha_1}, 0, ....)$ and $v=(0,e^{i\beta_2},...)$ for some $\alpha_1,\beta_1\in\mathbb{R}$. Since $\Phi$ does phase retrieval, we have that $u+v$ has the same phase as $e^{i\gamma}(u-v)$, with some real constant $\gamma$.  In particular $\arg (u_1+v_1)=\arg  (e^{i\gamma}(u_1-v_1))$, i.e. $\arg (e^{i\alpha_1})=\arg  (e^{i\gamma}e^{i\alpha_1})$. Similarly $\arg (u_2+v_2)=\arg  (e^{i\gamma}(u_2-v_2))$, i.e. $\arg (e^{i\beta_2})=\arg  (-e^{i\gamma}e^{i\beta_2})$. However the first condition implies $\gamma=0$ and the second gives $\gamma=\pi$, a contradiction.
\vskip12pt
\noindent {\bf Case 1}:    $u_jv_j\neq 0$  for some $1\leq j\leq N$.
\vskip12pt
Without loss of generality, we can assume $u_1v_1\neq 0$ and by multiplying by the appropriate constants we may also assume $|u_1|=|v_1|=r_1>0$. Then by Lemma \ref{Lem5.11}, for each $1\leq j\leq N$ we have that
\[\tan (\gamma)=\frac{2\Im (\overline{u_j}v_j)}{|u_j|^2-|v_j|^2}.\]
By assumption $|u_1|=|v_1|$, therefore $\gamma=\pi/2$ and hence $|u_j|=|v_j|$ for all $1\leq j\leq N$.  So we have shown that
$$u=(r_1e^{i\alpha_1},r_2e^{i\alpha_2},\ldots,r_Ne^{i\alpha_N})$$
and 
$$v=(r_1e^{i\beta_1},r_2e^{i\beta_2},\ldots,r_Ne^{i\beta_N}).$$
 Now we claim that $\sin (\beta_j-\alpha_j)=c$ for all $j$. To see this note that since $\arg(2u_j+v_j)=\arg (e^{i\theta} (2u_j-v_j))$  for  all $j$ and fixed $\theta$, then by Lemma \ref{Lem5.11}  we see that
\[c=\tan \theta=\frac{4\Im (\overline{u_j}v_j)}{3r_j^2}=\frac{4}{3}\sin(\beta_j-\alpha_j)\;\;\;\; \forall 1\leq j\leq N.\]
For each $j$, set $a_j=\cos (\beta_j-\alpha_j)=\pm\sqrt{1-c^2}$. We can express $v=w+ciu$ where
$$w=(a_1r_1e^{i\alpha_1},a_2r_2e^{i\alpha_2},\ldots,a_Nr_Ne^{i\alpha_N}).$$
Now we rewrite $$v=\left(r_1e^{i\alpha_1}e^{i(\beta_1-\alpha_1)},r_2e^{i\alpha_2}e^{i(\beta_2-\alpha_2)},
\ldots,r_Me^{i\alpha_M}e^{i(\beta_M-\alpha_M)}\right)$$
 and each $e^{i(\beta_j-\alpha_j)}=\cos (\beta_j-\alpha_j)+i\sin(\beta_j-\alpha_j)=a_j+ic$. We must show $w=0$. Recall that for every $n$ we have 
 $$0=\langle \phi_n\phi_n^*u,w+ciu\rangle_\RR=
 \langle \phi_n\phi_n^*u,w\rangle_\RR+\langle \phi_n\phi^*_nu, ciu\rangle_\RR.$$
 By Lemma \ref{Lem5.2} we see that $\langle \phi_n\phi_n^*u,w\rangle_\RR=0$ for all $n$. Note that $w=0$ if and only if $a_j=0$ for all $j$. This is clear since that if $a_{1}\neq 0$ then the first component of $a_1u+w$ is non-zero but the the first component of $a_1u-w$ is $0$ (assuming $u_1\neq 0$) which contradicts to $w=0$.

\end{proof}

\section{Weak Phase Retrieval}

We weaken the notion of phase retrieval.

\begin{definition}
Two vectors in $\HH_m$,
$x=(a_1,a_2,\ldots,a_m)$ and $y=(b_1,b_2,\ldots,b_m)$
{\bf weakly} have the same phase if
there is a $|\theta|=1$ so that
\[ \mbox{phase}(a_i)= \theta \mbox{phase}(b_i), \mbox{ for all }i=1,2,\ldots,m, \mbox{ for which } a_i\not= 0 \not= b_i.\]
In the real case, if $\theta=1$ we say $x,y$ {\bf weakly have the same signs} and if $\theta=-1$ they {\bf weakly have opposite signs.}
\end{definition}

\begin{definition}
A family of vectors $\{\phi_i\}_{i=1}^n$ in $\HH_m$
does {\bf weak phase retrieval} if for any
$x=(a_1,a_2,\ldots,a_m)$ and $y=(b_1,b_2,\ldots,b_m)$
in $\HH_m$, with
\[ |\langle x,\phi_i\rangle|=|\langle y,\phi_i\rangle|,
\mbox{ for all }i=1,2,\ldots,m,\]
there is a $|\theta|=1$ so that
\[ \mbox{phase}(a_i)= \theta \mbox{phase}(b_i), \mbox{ for all }i=1,2,\ldots,m, \mbox{ for which } a_i\not= 0 \not= b_i.\]
\end{definition}

The difference with {\it phase retrieval} is that
we are now allowing $a_i =0$ and $b_i \not= 0$.

An example of weak phase retrieval which does not yield
phase retrieval in $\RR^m$ is given by:

Let $\Phi=\{\phi_i\}_{i=1}^{m+1}$ be the column vectors
of the matrix:
\[A=\begin{bmatrix}

1&-1&1\cdots&1&1\\
1&1&-1\cdots &1&1\\
\vdots&\vdots & \vdots& \vdots & \vdots \\
1&1&1 \cdots &1&1
\end{bmatrix}_{m \times (m+1)}\]

Then for any $x=(a_1,a_2,\ldots,a_m)$ and $y=(b_1,b_2,\ldots, b_m)$, if

\[ |\langle x,\phi_i\rangle|^2 = |\langle y,\phi_i\rangle|^2,\]
then by expanding out and subtracting rows from each other, we will find that:
\[ a_ia_j=b_ib_j,\mbox{ for all }i\not= j.\]

This family of (m+1)-vectors in $\RR_m$ does weak phase
retrieval.  To see this, we need a proposition.  Notice
that there are too few vectors here to do phaseless
reconstruction.

\begin{proposition}\label{prop1}
Let $x=(a_1,a_2,\ldots,a_m)$ and $y=(b_1,b_2,\ldots,b_m)$
in $\RR_m$.  The following are equivalent:
\begin{enumerate}
\item We have
\[sgn\ (a_ia_j)=sgn\ (b_ib_j),\mbox{ for all }1\le i\not= j \le m.\]
\item Either $x,y$ have weakly the same signs or they have weakly opposite
signs
\end{enumerate}
\end{proposition}

\begin{proof}
$(1) \Rightarrow (2)$:
Let
\[ I = \{1\le i \le m:a_i=0\}\mbox{ and }
J=\{1\le i \le n:b_i=0\}.\]
Let
\[ K=[m]\setminus (I \cup J).\]
So $i\in K$ if and only if $a_i \not= 0 \not= b_i$.
Let $i_0 = min\ K$.  We examine two cases:
\vskip12pt
\noindent {\bf Case 1}:  $sgn\ a_{i_0} = sgn\ b_{i_0}$.
\vskip12pt
For any $i_0 \not= k \in K$, $a_{i_0}a_k=b_{i_0}b_k$,
implies $sgn\ a_k = sgn\ b_k$.  Since all other coordinates
of either $x$ or $y$ are zero, it follows that $x,y$
weakly have
the same signs.

\vskip12pt
\noindent {\bf Case 2}:  $sgn\ a_{i_0}=-sgn\ b_{i_0}$.
\vskip12pt
For any $i_o \not= k \in K$, $a_{i_0}a_k = b_{i_0}b_k$
implies $sgn\ a_k = - sgn\ b_k$.  Again, since all other
coordinates of either $x$ or $y$ are zero, it follows
that $x,y$ weakly have opposite signs.
\vskip12pt
$(2)\Rightarrow (1)$:  This is obvious.
\end{proof}

\begin{theorem}\label{thm1}
Let $x=(a_1,a_2,\ldots,a_m)$ and $y=(b_1,b_2,\ldots,b_m)$
in $\RR_m$ and assume we have:
\[ a_ia_j=b_ib_j,\mbox{ for all }1\le i\not= j \le m.\]
Then:
\begin{enumerate}
\item Either $x,y$ have weakly the same signs or they have weakly the opposite
signs.

\item One of the following holds:

(i)  There is a $1\le i \le m$ so that $a_i=0$ and $b_j=0$ for all $j\not= i$.

(ii)  There is a $1\le i \le m$ so that $b_i=0$ and $a_j=0$
for all $j\not= i$.

(iii)  If (i) and (ii) fail and
\[ I = \{1\le i \le m:a_i\not= 0 \not= b_i\},\]
then the following hold:

\ \ \ \ (a)  If $i\in I^c$ then $a_i=b_i=0$.

\ \ \ \ (b)  For all $i\in I$, $|a_i|=|b_i|$.
\end{enumerate}
\end{theorem}

\begin{proof}
(1)  This follows from Proposition \ref{prop1}.
\vskip12pt
(2)  (i)  Assume $a_i=0$ but $b_i\not= 0$.  Then for all $j\not= i$ we have
$a_ia_j=0=b_ib_j$ and so $b_j=0$.

(ii)  This is symmetric to (i).

(iii)  If (i) and (ii) fail, then by definition, for any
$i$, either both $a_i$ and $b_i$ are zero or they are both
non-zero, which proves (a).

Fix $i\in I$.  Choose any $j\not= k \in I
\setminus \{i\}$.
Then
\[ a_ia_j=b_ib_j\mbox{ and } a_ia_k=b_ib_k.\]
Multiplying the left-hand-sides and the right-hand-sides
yields,
\[ a_i^2a_ja_k = b_i^2b_jb_k.\]
Since $a_j,a_k,b_j,b_k$ are all non-zero and
$a_ja_k=b_jb_k$, we have that $a_i^2=b_i^2$.

\end{proof}

\end{document}